\documentclass[10pt]{amsart}
\usepackage[utf8]{inputenc}

\usepackage{amsmath,amssymb,amsthm}
\usepackage{mathrsfs}
\usepackage{enumerate}
\usepackage{graphicx,xcolor}
\usepackage[hidelinks]{hyperref}

\newcommand{\dd}{\mathrm{d}}
\newcommand{\E}{\mathbb{E}}

\newcommand{\R}{\mathbb{R}}

\newcommand{\scal}[2]{\left\langle #1, #2 \right\rangle}
\newcommand{\red}{}

\renewcommand{\leq}{\leqslant}
\renewcommand{\geq}{\geqslant}
\newcommand{\ls}{\leqslant}
\newcommand{\gr}{\geqslant}

\DeclareMathOperator{\Var}{Var}

\usepackage{xpatch}
\makeatletter
\def\thm@space@setup{%
  \thm@preskip=12pt plus 0pt minus 0pt
  \thm@postskip=0pt plus 0pt minus 0pt
}
\xpatchcmd{\proof}{6\p@\@plus6\p@\relax}{\z@skip}{}{}
\makeatother

\newtheorem{theorem}{Theorem}
\newtheorem{lemma}[theorem]{Lemma}
\newtheorem{corollary}[theorem]{Corollary}

\theoremstyle{remark}
\newtheorem{remark}[theorem]{Remark}

\theoremstyle{definition}

\title{Sharp bounds on $p$-norms for sums of independent uniform random variables, $0 < p < 1$}

\author{Giorgos Chasapis}

\author{Keerthana Gurushankar}

\author{Tomasz Tkocz}

\address{Carnegie Mellon University; Pittsburgh, PA 15213, USA.}

\email{\{gchasapi,  kgurusha,  ttkocz\}@andrew.cmu.edu}

\thanks{TT's research supported in part by NSF grant DMS-1955175.}

\date{\today}

\begin{document}

\begin{abstract} 
We provide a sharp lower bound on the $p$-norm of a sum of independent uniform random variables in terms of its variance when $0 < p < 1$. We address an analogous question for $p$-R\'enyi entropy for $p$ in the same range.
\end{abstract}

\maketitle

\bigskip

\begin{footnotesize}

\end{footnotesize}

\bigskip

\section{Introduction and results}

Moment comparison inequalities for sums of independent random variables, {\red that is} Khinchin-type inequalities, first established by Khinchin for Rademacher random variables (random signs) in his proof of the law of the iterated logarithm (see \cite{K}), have been extensively studied ever since his work. Particularly challenging, interesting and conducive to new methods is the question of sharp constants in such inequalities. We only mention in passing several classical as well as recent references, \cite{BC, H, HNT, Ko, KLO, LO-best, NO, S}. This paper finishes the pursuit of sharp constants in $L_p - L_2$ Khinchin inequalities for sums of independent uniform random variables, addressing the range $0 < p < 1$. We are also concerned with a $p$-R\'enyi entropy analogue.

\subsection{Moments}
Let $U_1, U_2, \dots$ be independent random variables uniform on $[-1,1]$. As usual, $\|X\|_p = (\E|X|^p)^{1/p}$ is the $p$-norm of a random variable $X$. Given $p > -1$, let $c_p$ and $C_p$ be the best constants such that for every integer $n \geq 1$ and real numbers $a_1, \dots, a_n$, we have
\begin{equation}\label{eq:khin}
c_p\left(\sum_{j=1}^n a_j^2\right)^{1/2} \leq \left\|\sum_{j=1}^n a_jU_j \right\|_p \leq C_p\left(\sum_{j=1}^n a_j^2\right)^{1/2},
\end{equation}
or in other words, since $\|\sum a_jU_j\|_2 = {\red \sqrt{\Var(\sum a_jU_j)}} = 3^{-1/2}\left(\sum a_j^2\right)^{\red 1/2}$,
finding $c_p$ and $C_p$ amounts to extremising the $p$-norm of the sum $\sum a_jU_j$ subject to a fixed variance,
\[
c_p = \inf  \left\|\sum_{j=1}^n a_jU_j \right\|_p, \qquad  C_p = \sup \left\|\sum_{j=1}^n a_jU_j \right\|_p,
\]
where the infimum and supremum and taken over all integers $n \geq 1$ and unit vectors $a = (a_1, \dots, a_n)$ in $\R^n$.

For $p > 1$, the optimal constants $c_p, C_p$ were found by Lata\l a and Oleszkiewicz in \cite{LO} (see also \cite{ENT} for an alternative approach and \cite{BC, KK} for generalisations in higher dimensions). They read
\begin{equation}\label{eq:LO}
\begin{split}
c_p &= \|U_1\|_p = (1+p)^{-1/p}, \\
C_p &= \lim_{n\to\infty} \left\|\frac{U_1 + \dots + U_n}{\sqrt{n}}\right\|_p = \|Z\|_p/\sqrt{3}, \quad p > 1,
\end{split}
\end{equation}
where $Z$ here and throughout the text denotes a standard $N(0,1)$ Gaussian random variable. In fact stronger results are available (extremisers are known via Schur-convexity for each fixed $n$). 

For $-1 < p < 0$, the behaviour is complicated by a phase transition (similar to the case of random signs as established by Haagerup in \cite{H}). It has recently been proved in \cite{CKT} that
\begin{align*}
c_p &= \min\left\{\|Z\|_p/\sqrt{3}, \left\|U_1+U_2\right\|_p/\sqrt{2}\right\} = \begin{cases}\|Z\|_p/\sqrt{3}, & -0.793.. < p < 0, \\  \left\|U_1+U_2\right\|_p/\sqrt{2}, & -1 < p {\red \ls}  -0.793..,  \end{cases}
\end{align*}
and the limiting behaviour of $c_p$ as $p \to -1^+$ recovers Ball's celebrated cube slicing inequality from \cite{B}.

The fact that
\[
C_p = \|U_1\|_p, \qquad -1 < p < 1,
\]
follows easily from unimodality and Jensen's inequality (see, e.g. Proposition 15 in \cite{ENT}). 

Thus what is unknown is the optimal value of $c_p$ for $0 < p < 1$ and this paper fills out this gap. Our main result reads as follows.

\begin{theorem}\label{thm:cp}
For $0 < p < 1$, $c_ p = \|Z\|_p/\sqrt{3}$ {\red is} the best constant in \eqref{eq:khin}.
\end{theorem}

We record for future use that
\[
\|Z\|_p^p = \frac{1}{\sqrt{2\pi}}\int_{-\infty}^\infty |x|^pe^{-x^2/2} \dd x = \frac{2^{p/2}}{\sqrt{\pi}}\Gamma\left(\frac{1+p}{2}\right).
\]

\subsection{R\'enyi entropy}
For $p \in [0,\infty]$, the $p$-R\'enyi entropy of a random variable $X$ with density $f$ is defined as (see \cite{R}),
\[
h_p(X) = \frac{1}{1-p}\log\left(\int_{\R} f^p\right)
\]
with $p \in \{0,1,\infty\}$ defined by taking the limit: $h_0(f) = \log|\text{supp}(f)|$ is the logarithm of the Lebesgue measure of the support of $f$, $h_1(f) = -\int f\log f$ is the Shannon entropy, and $h_\infty = -\log\|f\|_\infty$, where $\|f\|_\infty$ is the $\infty$-norm of $f$ (with respect to Lebesgue measure). The question of maximising R\'enyi entropy under a variance constraint (or more generally, a moment constraint) for general distributions has been fully understood and leads to the notion of \emph{relative entropy} that is of importance in information theory, providing a natural way of measuring distance to the extremal distributions (see \cite{CHV, JV, LYZ, M}). In analogy to Theorem \ref{thm:cp}, we provide an answer for $p$-R\'enyi entropies, $0 < p < 1$, for sums of uniforms under the variance constraints.
\begin{theorem}\label{thm:ent}
Let $0 < p < 1$. For every unit vector $a = (a_1, \dots, a_n)$, we have
\[
h_p\left(U_1\right) \leq h_p\left(\sum_{j=1}^n a_jU_j\right) \leq h_p\left(Z/\sqrt{3}\right).
\]
\end{theorem}
The lower bound is a simple consequence of the entropy power inequality. The upper bound is interesting in that the maximizer among \emph{all} distributions of fixed variance is \emph{not} Gaussian (rather, with density proportional to $(1+x^2)^{-1/(1-p)}$ for $\frac{1}{3} < p < 1$ and it does not exist for $p < \frac{1}{3}$, see e.g. \cite{CHV}). It is derived from the $L_q - L_2$ Khinchin inequality for even $q$.

\subsection{Organisation of the paper}
In Section \ref{sec:proof} we give an overview of the proof of Theorem \ref{thm:cp} and show a reduction to two main steps: an integral inequality and an inductive argument. Then in Section \ref{sec:lemmas} we gather all technical lemmas needed to accomplish these steps which is then done in Sections \ref{sec:intineq} and \ref{sec:ind}, respectively. Section \ref{sec:ent} contains a short proof of Theorem \ref{thm:ent}.

\subsection*{Acknowledgments.} 
We should very much like to thank Alexandros Eskenazis for the stimulating correspondence. {\red We are also indebted to anonymous referees for many valuable comments which helped significantly improve the manuscript.}

\section{Proof of the main result}\label{sec:proof}

\subsection{Overview}
We follow an approach developed by Haagerup in \cite{H}, with major simplifications advanced later by Nazarov and Podkorytov in \cite{NP}. In essence, the argument begins with a Fourier-analytic integral representation for the power function $|\cdot|^p$ which allows to take advantage of independence and in turn, by virtue of the AM-GM inequality, to reduce the problem to establishing a certain integral inequality involving the Fourier transforms of the uniform and Gaussian distributions. Since this inequality holds only in a specific range of parameters, additional arguments are needed, mainly an induction on the number of summands $n$ (similar problems were faced in e.g. \cite{CKT, NP, Ko}). In our case, this is further complicated by the fact that the base of the induction fails for large values of $p$ (roughly for $p > 0.7$).

{\red
\begin{remark}
We point out that the main difference between the regimes $p {\red \gr} 1$ and $p < 1$ is that for the former convexity type arguments allow to establish stronger comparison results, namely the Schur-convexity/concavity of the function
\[
(\sqrt{x_1}, \dots, \sqrt{x_n}) \mapsto \E\left|\sum_{j=1}^n \sqrt{x_j}U_j\right|^p.
\]
By combining Theorems 2 and 3 of \cite{BC} (see also (6.1) therein), a necessary condition for this is the concavity/convexity of the function $x \mapsto \E|U_1 + \sqrt{x}|^p$. The calculations following Corollary 1 in the same work show that this is the case only for $p \geq 1$. In other words, when $p<1$, the function above is neither Schur-convex nor Schur-concave and the Fourier-analytic approach seems to be indispensable.
\end{remark}
}

\subsection{Details}
The aforementioned Fourier-analytic formula reads as follows (it can be found for instance in \cite{H}, but we sketch its proof for completeness).

\begin{lemma}\label{lm:fourier}
Let $0 < p < 2$ and $\kappa_p = \frac{2}{\pi}\Gamma(1+p)\sin\left(\frac{\pi p}{2}\right)$. For a random variable $X$ in $L_2$ with characteristic function $\phi_X(t) = \E e^{itX}$, we have
\[
\E|X|^p = \kappa_p \int_0^\infty \frac{1 - \mathrm{Re} \phi_X(t)}{t^{p+1}} \dd t.
\]
\end{lemma}
\begin{proof}
A change of variables {\red establishes} $|x|^p = \kappa_p \int_0^\infty \frac{1-\cos(tx)}{t^{p+1}} \dd t$, $x \in \R$. {\red We then apply this} to $X$ and take the expectation.
\end{proof}

We begin the proof of Theorem \ref{thm:cp}.
Let $0 < p < 1$ and $c_p = \|Z\|_p/\sqrt{3}$.  Let $a_1, \dots, a_n$ be {\red (without loss of generality)} nonzero real numbers with $\sum_{j=1}^n a_j^2 = 1$. {\red By symmetry of the uniform distribution we assume without loss of generality that} they are in fact positive. From Lemma \ref{lm:fourier}, we obtain
\[
\E\left|\sum_{j=1}^n a_jU_j\right|^p = \kappa_p\int_0^\infty \frac{1-\prod_{j=1}^n \phi(a_jt)}{t^{1+p}} \dd t,
\]
where we have used independence and put $\phi(t) = \E e^{itU_1} = \frac{\sin t}{t}$ to be the characteristic function of the uniform distribution. We seek a sharp lower-bound on this expression (attained when $a_1 = \dots = a_n = \frac{1}{\sqrt{n}}$ and $n \to \infty$, as anticipated by Theorem \ref{thm:cp}). By the AM-GM inequality,
\[
\left|\prod_{j=1}^n \phi(a_jt)\right| \leq \sum_{j=1}^n a_j^2|\phi(a_jt)|^{1/a_j^{2}}.
\]
As a result,
\[
\E\left|\sum_{j=1}^n a_jU_j\right|^p \geq \sum_{j=1}^n a_j^2\mathcal{I}_p(1/a_j^{2}),
\]
where we have set
\[
\mathcal{I}_p(s) = \kappa_p\int_0^\infty \frac{1-\left|\frac{\sin(t/\sqrt{s})}{t/\sqrt{s}}\right|^s}{t^{p+1}} \dd t, \qquad s \geq 1.
\]
Note that $\frac{\sin(t/\sqrt{s})}{t/\sqrt{s}} = 1 - \frac{t^2}{6s} + O(1/s^2)$ for a fixed $t$ as $s \to \infty$ and consequently,
\[
\mathcal{I}_p(\infty) = \lim_{s\to\infty} \mathcal{I}_p(s) = \kappa_p\int_0^\infty \frac{1-e^{-t^2/6}}{t^{p+1}} \dd t = \E|Z/\sqrt{3}|^p,
\]
where the last equality follows from Lemma \ref{lm:fourier} because $e^{-t^2/6}$ is the characteristic function of $Z/\sqrt{3}$, $Z \sim N(0,1)$ (the exchange of the order of the limit and integration in the second equality can be easily justified by truncating the integral, see, e.g., (15) in \cite{CKT}). In particular, if for some $p$ and $s_0$,
\begin{equation}\label{eq:Ip}
\mathcal{I}_p(s) \geq \mathcal{I}_p(\infty), \qquad \text{for all } s \geq s_0,
\end{equation}
then
\begin{equation}\label{eq:cor}
\E\left|\sum_{j=1}^n a_jU_j\right|^p \geq \E|Z/\sqrt{3}|^p = c_p^p,
\end{equation}
as long as $1/a_j^2 \geq s_0$ for each $j$. If \eqref{eq:Ip} were true for all $0 < p < 1$ with $s_0 = 1$, then the proof of Theorem \ref{thm:cp} would be complete. Unfortunately, that is not the case. In Section \ref{sec:intineq} we show the following result.

\begin{theorem}\label{thm:s>1}
Inequality \eqref{eq:Ip} holds for every $0.6 < p < 1$ with $s_0 = 1$.
\end{theorem}

As a result, when $0.6 < p < 1$, \eqref{eq:cor} holds for arbitrary $a_j$ and the proof of Theorem~\ref{thm:cp} is complete in this case. For smaller values of $p$, $s_0$ has to be increased.

\begin{theorem}\label{thm:s>2}
Inequality \eqref{eq:Ip} holds for every $0 < p < 1$ with $s_0 = 2$.
\end{theorem}

This is proved in Section \ref{sec:intineq}. Consequently, \eqref{eq:cor} holds provided that $a_j^2 \leq \frac{1}{2}$ for each $j$. To remove this restriction, we employ an inductive argument of Nazarov and Podkorytov from \cite{NP} developed for random signs and adapted to the uniform distribution in \cite{CKT}. This works for $0 < p < 0.69$ and the proof of Theorem \ref{thm:cp} is complete. This is done in Section \ref{sec:ind}.

\section{Auxiliary lemmas}\label{sec:lemmas}

To show Theorems \ref{thm:s>1} and \ref{thm:s>2} and carry out the inductive argument, we first prove some technical lemmas.

\subsection{Lemmas concerning the sinc function}

The zeroth spherical Bessel function (of the first kind) $j_0(x) = \frac{\sin x}{x} = \text{sinc}(x)$ is sometimes referred to as the sinc function. As the characteristic function of a uniform random variable, it plays a major role in our approach. We shall need several elementary estimates.

\begin{lemma}\label{lm:f<g}
For $0 < t < \pi$, we have $\frac{\sin t}{t} < e^{-t^2/6}$.
\end{lemma}
\begin{proof}
This follows from the product formula, $\frac{\sin t}{t} = \prod_{n=1}^\infty \left(1 - \frac{t^2}{n^2\pi^2}\right)$. Since each term is positive for $0 < t < \pi$, the lemma follows by applying $1 + x \leq e^x$ and $\sum_{n=1}^\infty \frac{1}{n^2} = \frac{\pi^2}{6}$.
\end{proof}

\begin{lemma}\label{lm:der}
$\sup_{t \in \R}\left|\cos t - \frac{\sin t}{t}\right| < \frac{11}{10}$.
\end{lemma}
\begin{proof}
Since both $\cos t$ and $\frac{\sin t}{t}$ are even, it suffices to consider positive $t$. By the Cauchy-Schwarz inequality, we have $\left|\cos t - \frac{\sin t}{t}\right| \leq \sqrt{1+\frac{1}{t^2}}$, so it suffices to consider $t < \frac{10}{\sqrt{21}}$. On $(0, \frac{\pi}{2})$, we have $\left|\cos t - \frac{\sin t}{t}\right| = \frac{\sin t}{t} - \cos t < 1 +  0 = 1$, so it remains to consider $\frac{\pi}{2} < t < \frac{10}{\sqrt{21}}$. Letting $t = \frac{\pi}{2} + x$, we have for $0 < x < \frac{10}{\sqrt{21}} - \frac{\pi}{2}$,
\[
\left|\cos t - \frac{\sin t}{t}\right| = \frac{\sin t}{t} - \cos t = \frac{\cos x}{x+\pi/2} + \sin x < \frac{1}{x+\pi/2} + x.
\]
Examining the derivative, the right hand side is clearly increasing, so it is upper bounded by its value at $x =  \frac{10}{\sqrt{21}} - \frac{\pi}{2}$ which is $\frac{\sqrt{21}}{10} + \frac{10}{\sqrt{21}} - \frac{\pi}{2} < 1.07$.
\end{proof}

\begin{lemma}\label{lm:maxima}
Let $k \geq 0$ be an integer and let $y_k$ be the value of the unique local maximum of $\left|\frac{\sin t}{t}\right|$ on $(k\pi,(k+1)\pi)$. Then
\[
\frac{1}{(k+1/2)\pi} \leq y_k \leq \frac{1}{k\pi}.
\]
Moreover, $y_1 < e^{-3/2}$.
\end{lemma}
\begin{proof}
The lower bound follows from taking $t = (k+1/2)\pi$, whereas the upper bound follows from $|\sin t| \leq 1$ and $t > k\pi$. The bound on $y_1$ is equivalent to $\sin t < e^{-3/2}(t+\pi)$, $0 < t < \pi/2$. To show this in turn, it suffices to upper bound $\sin t$ by its tangent at, e.g., $t = 1.3$.
\end{proof}

\begin{lemma}\label{lm:t0}
For $y \in (0,\frac{1}{30\pi})$, let $t = t_0$ be the unique solution to $\frac{\sin t}{t} = y$ on $(0,\pi)$. Then $t_0 > 0.98\pi$. Let $t = t_1$ be the larger of the two solutions to $\frac{|\sin t|}{t} = y$ on $(\pi,2\pi)$. Then $t_1 > 1.97\pi$.
\end{lemma}
\begin{proof}
Note that
$\frac{\sin t_0}{t_0} = y < \frac{1}{30\pi} < \frac{\sin (0.98\pi)}{0.98\pi}$.
Since $\frac{\sin t}{t}$ is decreasing on $(0,\pi)$, it follows that $t_0>0.98\pi$. Similarly, we check that $\frac{|\sin(1.97\pi)|}{1.97\pi} > \frac{1}{30\pi}$ to justify the claim about $t_1$.
\end{proof}

\begin{lemma}\label{lm:sin}
For $0 < x < \pi$,
\[
\frac{1}{\sin^2 x} > \frac{1}{x^2}  + \frac{1}{(\pi-x)^2}.
\]
\end{lemma}
\begin{proof}
It is well known (and follows from $\sin(2x) = 2\sin x \cos x$) that 
\[
\frac{\sin x}{x} = \prod_{k=1}^\infty \cos(x/2^k).
\]
In particular, for $0 < x < \pi$, we have $\left(\frac{\sin x}{x}\right)^2 < \cos^2(x/2)$, hence
\begin{align*}
\sin^2 x\left(\frac{1}{x^2}  + \frac{1}{(\pi-x)^2}\right) &= \left(\frac{\sin x}{x}\right)^2 + \left(\frac{\sin (\pi-x)}{\pi - x}\right)^2 \\
&< \cos^2(x/2) + \cos^2(\pi/2-x/2) = 1.
\end{align*}
\end{proof}

\begin{lemma}\label{lm:slopes}
Let $k \geq 1$ be an integer. On $((k-1)\pi,k\pi)$, we have

(i) the function $\frac{|\sin t|}{t(t-(k-1)\pi)}$ is nonincreasing,

(ii) the function $\frac{|\sin t|}{t(k\pi - t)}$ is unimodal (first increases and then decreases).
\end{lemma}
\begin{proof}
(i) The derivative equals
\[
\frac{|\sin t|}{t(t-(k-1)\pi)}\left(\cot(t) - \frac{1}{t} - \frac{1}{t-(k-1)\pi}\right)
\]
which is negative on $((k-1)\pi,k\pi)$ because on this interval, $\cot(t) < \frac{1}{t-(k-1)\pi}$ (as, by periodicity, being equivalent to $\cot(t) < \frac{1}{t}$ on $(0,\pi)$, which is clear -- recall that $\tan(x) > x$ on $(0,\frac{\pi}{2})$).

(ii) Here, the derivative reads
\[
\frac{|\sin t|}{t(k\pi - t)}h(t), \qquad h(t) = \cot(t) - \frac{1}{t} + \frac{1}{k\pi - t}.
\]
We shall argue that $h(t)$ is decreasing on $((k-1)\pi,k\pi)$. This suffices, since $h(t) > 0$ for $t$ near $(k-1)\pi$ and $h(t) < 0$ for $t$ near $k\pi$. Setting $t = (k-1)\pi+x$, we have
\begin{align*}
h'(t) &= -\frac{1}{\sin^2 t} + \frac{1}{t^2} + \frac{1}{(k\pi-t)^2} \\
&\leq -\frac{1}{\sin^2 x} + \frac{1}{x^2} + \frac{1}{(\pi-x)^2} < 0,
\end{align*}
by Lemma \ref{lm:sin}.
\end{proof}

\subsection{Lemmas concerning sums of \emph{p}-th powers}

Our computations require several technical bounds on various expressions involving sums of $p$-th powers.

\begin{lemma}\label{lm:der-p}
Let $0 < p < 1$ and let $1 \leq m \leq 29$ be an integer. Set
\[
u_m(p) = B_m\left(b_{0,m}^p + 2\sum_{k=1}^m b_{k,m}^p\right)
\]
with
\[
B_m = \frac{20\log\Big(\pi(m+3/2)\Big)}{11\pi(m+3/2)}, \qquad b_{k,m} = \frac{1}{k+1}\sqrt{\frac{6}{\pi^2}\log\Big(\pi(m+3/2)\Big)}.
\]
Then, 
\[
u_m(p) > 1.
\]
\end{lemma}
\begin{proof}
Fix $m$. Plainly, $u_m(p)$ is a convex function (as a sum of convex functions). Thus, $u_m'(p) < u_m'(1)$ for $0 < p < 1$. We have,
\[
u_m'(1) = B_m\left(b_{0,m}\log b_{0,m} + 2 \sum_{k=1}^m b_{k,m}\log b_{k,m}\right)
\]
and Table \ref{tab:u} shows that each $u_m'(1)$ is negative, so each $u_m$ is decreasing. Therefore, $u_m(p) > u_m(1)$, for $0 < p < 1$ and Table \ref{tab:u} shows that each $u_m(1)$ is greater than $1$. This finishes the proof.
\end{proof}

\begin{table}[!hb]
\begin{center}
\caption{Lower bounds on the values of  $-u_m'(1)$ and $u_m(1)$.}
\label{tab:u}
\begin{tabular}{r|ccccccccccccccc}
$m$ & 1 & 2 & 4 & 4 & 5 & 6 & 7 & 8 & 9 & 10 \\\hline
$-u_m'(1)$ & $0.24$ & $0.44$ & $0.58$ & $0.70$ & $0.79$ & $0.86$ & $0.91$ & $0.96$ & $0.99$ & $1.02$  \\
$u_m(1)$ & $1.06$ & $1.27$ & $1.36$ & $1.40$ & $1.41$ & $1.41$ & $1.40$ & $1.38$ & $1.36$ & $1.34$
\end{tabular}

\bigskip

\bigskip

\begin{tabular}{r|ccccccccccccccc}
$m$ & 11 & 12 & 13 & 14 & 15 & 16 & 17 & 18 & 19 & 20 \\\hline
$-u_m'(1)$ & $1.05$ & $1.07$ & $1.08$ & $1.10$ & $1.11$ & $1.12$ & $1.13$ & $1.13$ & $1.14$ & $1.14$ \\
$u_m(1)$ & $1.32$ & $1.29$ & $1.27$ & $1.25$ & $1.23$ & $1.21$ & $1.19$ & $1.17$ & $1.15$ & $1.14$ 
\end{tabular}

\bigskip

\bigskip

\begin{tabular}{r|ccccccccccccccc}
$m$ & 21 & 22 & 23 & 24 & 25 & 26 & 27 & 28 & 29  \\\hline
$-u_m'(1)$ & $1.14$ &$1.14$ & $1.15$ & $1.15$ & $1.15$ & $1.15$ & $1.15$ & $1.15$ & $1.14$ &\\
$u_m(1)$ & $1.12$ & $1.10$ & $1.09$ & $1.07$ & $1.06$ & $1.04$ & $1.03$ & $1.02$ & $1.00$
\end{tabular}
\end{center}
\end{table}

\begin{lemma}\label{lm:sum-p}
For $0 < p < 1$, let
\begin{align*}
\alpha_p &= 2\pi^{-p+1}\left(3-\frac{1}{1-p}+\frac{3}{2p}\right),  \qquad
\beta_p = \frac{2}{1-p} + \frac{1.05^p}{p},  \qquad \gamma_p = \frac{3\pi}{p}, \\
\delta_p &= \frac{1}{p}\left(\frac{30\pi}{6\log(30\pi)}\right)^{p/2}
\end{align*}
and 
\[
h_p(y) = \delta_p y^{\frac{p}{2}-1} + \gamma_p y^p - \beta_py^{p-1} -\alpha_p.
\]
Then $h_p(y) > 0$ for every $0 < y < \frac{1}{30\pi}$.
\end{lemma}
\begin{proof}
Plainly, it suffices to show the following two claims,
\begin{align}\label{eq:hp'}
h_p'(y) &< 0, \qquad 0 < y < \frac{1}{30\pi},  \\\label{eq:hp}
h_p\left(\frac{1}{30\pi}\right) &> 0. 
\end{align}
To prove \eqref{eq:hp'}, first we find
\[
y^{2-\frac{p}{2}}h_p'(y) = -\left(1 - \frac{p}{2}\right)\delta_p + p\gamma_py^{\frac{p}{2}+1} + (1-p)\beta_p y^{\frac{p}{2}}
\]
which is clearly increasing in $y$, thus to show that it is negative, it suffices to prove that at $y = \frac{1}{30\pi}$, which in turn is equivalent to 
\[
2.1p + (1-p)1.05^p < \left(1-\frac{p}{2}\right)\left(\frac{30\pi}{\sqrt{6\log(30\pi)}}\right)^{p}.
\]
Crudely, $(1-p)1.05^p < 1.05^p < 1+0.05p$, by convexity, thus it suffices to show that
\[
1+2.15p < \left(1-\frac{p}{2}\right)A^p,
\]
where we put $A = \frac{30\pi}{\sqrt{6\log(30\pi)}}$. Equivalently, after taking the logarithm, the inequality becomes
\[
p\log A + \log \left(1-\frac{p}{2}\right) - \log(1+2.15p) > 0.
\]
Note that at $p = 0$ this becomes equality. We claim that the derivative of the left hand side is positive for $0 < p < 1$, which will finish the argument. The derivative is $\log A - \frac{1}{2-p} - \frac{2.15}{1+2.15p}$ which is clearly concave, thus it suffices to examine whether it is positive at the end-points $p=0$ and $p=1$, which respectively becomes $\log A > 2.65$ and $\log A > 1 + \frac{2.15}{3.15}$. Since $\log A = 2.89..$, both are clearly true.

It remains to show \eqref{eq:hp}, that is that the following is positive for every $0 < p < 1$,
\[
30^p\pi^{p-1}h_p\left(\frac{1}{30\pi}\right) =\underbrace{30\frac{\left(\frac{30\pi}{\sqrt{6\log(30\pi)}}\right)^p - 1.05^p}{p}}_{L(p)} - \Big(\underbrace{2\frac{30-30^p}{1-p} + 3\frac{30^p-1}{p} + 6\cdot 30^p}_{R(p)}\Big).
\]
Both $L(p)$ and $R(p)$ are strictly increasing and convex on $(0,1)$. This is clear for $L$, since its Taylor expansion at $p=0$ has positive coefficients. Similarly for the term $\frac{30^p-1}{p}$ in $R(p)$. To see that $\frac{30-30^p}{1-p}$ is strictly increasing and convex, write it as $30\int_1^{30} u^p \frac{\dd u}{u^2}$. 

\emph{Case 1: $0 < p < 0.6$.} By convexity, using a tangent line
\[
L(p) \geq L(0.24) + L'(0.24)(p-0.24) = \ell(p)
\]
and a chord
\[
R(p) \leq \frac{p}{0.6}R(0.6) + \frac{0.6-p}{0.6}R(0^+) = r(p).
\]
With hindsight, the tangent and the chord are chosen such that $\ell > r$ on $(0,0.6)$, which can be checked directly by looking at the values of these linear functions at the end-points.

\emph{Case 2: $0.6 < p < 1$.} Similarly, by convexity, using a tangent line
\[
L(p) \geq L(0.8) + L'(0.8)(p-0.8) = \tilde\ell(p)
\]
and a chord
\[
R(p) \leq \frac{1-p}{0.4}R(0.6) + \frac{p-0.6}{0.4}R(1^-) = \tilde r(p).
\]
Again, with hindsight, the tangent and the chord are chosen such that $\tilde\ell > \tilde r$ on $(0.6,1)$. This completes the proof.
\end{proof}

\subsection{Lemmas concerning the gamma function}
For the inductive part of our argument, we will later need bounds on the following function
\[
\psi(p) = \frac{1+p}{\sqrt{\pi}}\left(\frac{4}{3}\right)^{p/2}\Gamma\left(\frac{1+p}{2}\right), \qquad 0 < p < 1.
\]
Recall the Weierstrass' product formula, $\Gamma(z) = \frac{e^{-\gamma z}}{z}\prod_{n=1}^\infty \left(1+\frac{z}{n}\right)^{-1}e^{z/n}$, where $\gamma = 0.57..$ is the Euler-Mascheroni constant. Writing $\sqrt{\pi}$ as $\Gamma(\frac{1}{2})$, we obtain
\begin{equation}\label{eq:psi-prod}
\psi(p) = e^{\frac{p}{2}\left(\log(4/3)-\gamma\right)}\prod_{n=1}^\infty \left(1+\frac{p}{2n+1}\right)^{-1}e^{\frac{p}{2n}}.
\end{equation}

\begin{lemma}\label{lm:gamma-small-p}
For $0 < p < 0.69$, we have $\psi(p) < \frac{1}{2-(3/2)^{p/2}}$.
\end{lemma}
\begin{proof}
{\red We show that}
\[
f(p) =\log(2-(3/2)^{p/2}) +  \log \psi(p)
\]
is negative on $(0, 0.69)$. By virtue of \eqref{eq:psi-prod},
\[
f''(p) = -\frac{1}{2}\log^2\left(\frac{3}{2}\right)\frac{(3/2)^{p/2}}{(2-(3/2)^{p/2})^2}+\sum_{n=1}^\infty \frac{1}{(2n+1+p)^2} 
\]
This is plainly a decreasing function. Using $\sum_{n=1}^\infty \frac{1}{(2n+1+p)^2}  \geq \sum_{n=1}^\infty \frac{1}{(2n+2)^2} = \frac{\pi^2-6}{24}$, we get with $f''(0.9) > 0.007$, so $f$ is strictly convex on $(0,0.9)$. Checking that $f(0) = 0$ and $f(0.69) < -0.0001$ finishes the proof.
\end{proof}

\begin{lemma}\label{lm:gamma-all-p}
For $0 < p < 1$, we have $\psi(p) < 1 + \frac{p(p+1)}{6}$.
\end{lemma}
\begin{proof}
{\red We show that}
\[
f(p) = -\log\left(1 + \frac{p(p+1)}{6}\right) + \log\psi(p)
\]
is negative on $(0,1)$. Since $f(0) = 0$, it suffices to show that $f'(p) < 0$ on $(0,1)$. Using \eqref{eq:psi-prod}, we have
\[
f'(p) = -\frac{2p+1}{p^2+p+6} + \frac{1}{2}\left(\log(4/3)-\gamma\right) + \sum_{n=1}^\infty\left(\frac{1}{2n} - \frac{1}{2n+1+p}\right).
\]
Now, for $R(p) = -\frac{2p+1}{p^2+p+6} + \frac{1}{2}\left(\log(4/3)-\gamma\right)$, $R''(p) = \frac{(2p+1)(17-p^2-p)}{(p^2+p+6)^3} > 0$ on $(0,1)$, so $R(p)$ is convex on $(0,1)$. Let $S(p) = \sum_{n=1}^\infty\left(\frac{1}{2n} - \frac{1}{2n+1+p}\right)$. Plainly, this is a concave function. Thus, using tangents at $p=0$ and $p=1$, $S(p) \leq \min\{L_0(p),L_1(p)\}$ with $L_0(p) = S(0) + S'(0)p = (1-\log 2) + (\frac{\pi^2}{8}-1)p$ and $L_1(p) = S(1) + S'(1)(p-1)= \frac{1}{2} + \frac{\pi^2-6}{24}(p-1)$. We obtain the upper-bounds on $f'(p)$ by the convex functions $R(p) + L_0(p)$ and $R(p) + L_1(p)$. Examining the end-points we conclude that the former is negative on $(0,0.5)$ and the latter is negative on $(0.4,1)$. Thus $f'(p) < 0$ on $(0,1)$, as desired.
\end{proof}

\section{Integral inequality: proofs of Theorems \ref{thm:s>1} and \ref{thm:s>2}}\label{sec:intineq}

First observe that using the integral expression for $\mathcal{I}_p(\infty)$, inequality \eqref{eq:Ip} becomes
\begin{align}
0 \leq \mathcal{I}_p(s) - \mathcal{I}_p(\infty) &= \kappa_p \int_0^\infty \frac{e^{-t^2/6} - \left|\frac{\sin(t/\sqrt{s})}{t/\sqrt{s}}\right|^s}{t^{p+1}} \dd t \notag\\\label{eq:Ipp}
&= \kappa_ps^{-p/2}\int_0^\infty \frac{e^{-st^2/6}-\left|\frac{\sin t}{t}\right|^s}{t^{p+1}} \dd t.
\end{align}
To tackle such an inequality with an oscillatory integrand, we rely on the following extremely efficient and powerful lemma of Nazarov and Podkorytov from \cite{NP} (for the proof, see e.g. \cite{Ko}).

\begin{lemma}[Nazarov-Podkorytov, \cite{NP}]\label{lm:NP}
Let {\red $M\in(0,\infty]$} and $f,g:X\to[0,M]$ be any two measurable functions on a measure space $(X,\mu)$. Assume that the \textit{modified distribution functions}
\[
F(y) = \mu(\{x\in X: f(x)<y\}) \hspace{20pt}\hbox{ and } \hspace{20pt} G(y) = \mu(\{x\in X : g(x)<y\})
\]
of $f$ and $g$ respectively are finite for every $y\in(0,M)$. If there exists $y_*\in(0,M)$ such that $G(y)\gr F(y)$ for all $y\in(0,y_*)$, $G(y)\ls F(y)$ for all $y\in(y_*,M)$, then the function
\[
s \mapsto \frac{1}{sy_0^s}\int_X(g^s-f^s)\,d\mu
\]
is increasing on the set $\{s>0:g^s-f^s\in L^1(X,\mu)\}$.
\end{lemma}

In view of \eqref{eq:Ip}, \eqref{eq:cor} and \eqref{eq:Ipp}, Theorems \ref{thm:s>1} and \ref{thm:s>2} immediately follow from the following lemma.

\begin{lemma}
Let $f(t) = \left|\frac{\sin t}{t}\right|$, $g(t) = e^{-t^2/6}$, $t > 0$, and set
\[
H(p,s) = \int_0^\infty \frac{g(t)^s - f(t)^s}{t^{p+1}} \dd t.
\]
We have,

(a) $H(p,s) \geq 0$ for every $0 < p < 1$ and $s \geq 2$,

(b) $H(p,s) \geq 0$ for every $0.6 < p < 1$ and $s \geq 1$.
\end{lemma}
\begin{proof}
Fix $0 < p < 1$. {\red We examine} the modified distribution functions
\begin{align*}
F(y) &= \mu(t > 0, \ f(t) < y), \\
G(y) &= \mu(t > 0, \ g(t) < y), \qquad 0 < y < 1,
\end{align*}
where $\dd \mu(t) = t^{-p-1} \dd t$. It suffices to show that 
\begin{equation}\label{eq:goal}\tag{$\star$}
G-F \text{ changes sign exactly once on $(0,1)$ at some $y = y_*$ from $+$ to $-$.}
\end{equation}
Then Lemma \ref{lm:NP} gives that 
\[
s \mapsto \frac{1}{sy_*^s}H(p,s)
\]
is increasing on $(0,\infty)$. In particular, (a) and (b) result from the following claims whose  proofs we defer until the end of this proof.

{\bf Claim A.} $H(p,2) \geq 0$ for every $0 < p < 1$.

{\bf Claim B.} $H(p,1) \geq 0$ for every $0.6 < p < 1$.

Towards \eqref{eq:goal}, let $1 = y_0 > y_1 > y_2 > \dots$ be the consecutive {\red maximum values} of $f$.
On $(0,\pi)$, $f \leq g$ (Lemma \ref{lm:f<g}), so $G-F < 0$ on $(y_1,1)$. We plan to find $a \in (0, y_1)$ with the following two properties

(i) $(G-F)' < 0$ on $(a, y_1)$,

(ii) $G-F > 0$ on $(0,a)$.

This clearly suffices to conclude \eqref{eq:goal}. 

Fix $m \in \{1, 2, \dots\}$ and $y \in (y_{m+1}, y_m)$. Plainly,
\[
G(y) = \int_{\sqrt{-6\log y}}^\infty \frac{\dd t}{t^{p+1}} = \frac{1}{p}\left(-6\log y\right)^{-p/2}.
\]
Let $t_0^+ = t_0^+(y)$ be the unique solution to $f(t) = y$ on $(0,\pi)$ and for each $1 \leq k \leq m$, let $t_k^- < t_k^+$ be the unique solutions to $f(t) = y$ on $(k\pi, (k+1)\pi)$ ($t_k^\pm = t_k^\pm(y)$ are functions of $y$). We have,
\begin{equation}\label{eq:F-mu}
F(y) = \mu(t_0^+, t_1^-) + \mu(t_1^+,t_2^-) + \dots + \mu(t_{m-1}^+, t_m^-) + \mu(t_m^+,\infty).
\end{equation}
\begin{center}
\begin{figure}[htb]
\hspace*{-6em}\includegraphics[scale=1.2]{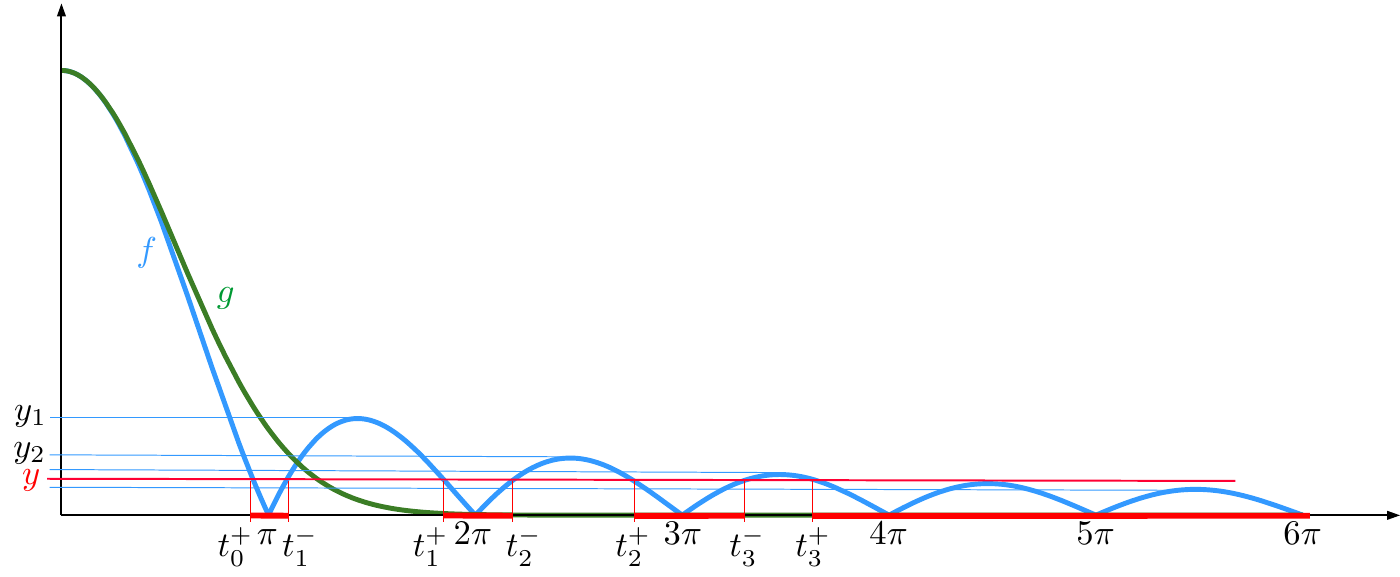}
\caption{Functions $f$, $g$ and the set $\{t > 0, f(t) < y\}$. Here $m=3$, i.e. $y_3 < y < y_4$.}
\end{figure}
\end{center}

\emph{Condition (i).} {\red Recall that $y\in(y_{m+1},y_m)$.} We have,
\[
G'(y) = \frac{3}{y}\left(-6\log y\right)^{-p/2-1}
\]
and, {\red differentiating \eqref{eq:F-mu} with respect to $y$ (using the fundamental theorem of calculus and chain rule),}
\[
F'(y) = \sum_{t: f(t) = y} \frac{1}{t^{p+1}|f'(t)|}.
\]
To lower bound $\frac{F'}{G'}$ in order to show that it is greater than $1$, we lower bound $F'$ and $\frac{1}{G'}$ separately as follows. First, using $|tf'(t)| = |\cos t - \frac{\sin t}{t}| < \frac{11}{10}$ for every $t > 0$ (Lemma \ref{lm:der}), we have, 
\[
F'(y) > \frac{10}{11}\sum_{t: f(t) = y} t^{-p} > \frac{10}{11}\pi^{-p}\left(1 + 2\sum_{k=1}^m (k+1)^{-p}\right),
\]
by crudely bounding $t_0^- < \pi$, $t_k^\pm < (k+1)\pi$. Second, since $y(-6\log y)^{p/2+1}$ is increasing on $(0,y_1)$ (it is increasing on $(0, e^{-1-p/2})$ and $e^{-1-p/2} > e^{-3/2} > y_1$), and $y_{m+1} > \frac{1}{\pi(m+3/2)}$ (Lemma \ref{lm:maxima}),
\begin{align*}
\frac{1}{G'(y)} = \frac{1}{3}y\left(-6\log y\right)^{p/2+1} &> \frac{1}{3}y_{m+1}\left(-6\log y_{m+1}\right)^{p/2+1} \\
&> \frac{1}{3}\frac{1}{\pi(m+3/2)}\left(6\log\Big(\pi(m+3/2)\Big)\right)^{p/2+1}.
\end{align*}
We obtain
\[
\frac{F'(y)}{G'(y)} > \frac{10}{33}\frac{1}{\pi^{p+1}(m+3/2)}\left(6\log\Big(\pi(m+3/2)\Big)\right)^{p/2+1}\left(1 + 2\sum_{k=1}^m (k+1)^{-p}\right).
\]
From Lemma \ref{lm:der-p} the right hand side is at least $1$ for every $0 < p < 1$ and $1 \leq m \leq 29$. Therefore, to guarantee that Condition (i) holds, we can choose any $a \geq y_{30}$.

We set $a = y_{30}$ and argue next that Condition (ii) holds for every $y \in (0,a)$. 

\emph{Condition (ii).}
We assume here that $m \geq 30$. Recall we have fixed $y \in (y_{m+1},y_m)$. Since $G$ is explicit, it suffices to upper bound $F$. We have,
\begin{align*}
F(y) &= \sum_{k=1}^m \int_{t_{k-1}^+}^{t_k^-} \frac{\dd t}{t^{p+1}} + \int_{t_m^+}^\infty \frac{\dd t}{t^{p+1}} \\
&\leq \sum_{k=1}^m (t_k^- - t_{k-1}^+)(t_{k-1}^+)^{-p-1} + \frac{1}{p}(t_m^+)^{-p}.
\end{align*}
For $k \geq 3$, we crudely estimate $t_{k-1}^+ \geq (k-1)\pi$, whereas for $k=1, 2$, we have $t_0^+ > 0.98\pi$ and $t_1^+ > 1.97\pi$, thanks to Lemma \ref{lm:t0}.
To upper bound the length $t_k^- - t_{k-1}^+$, note that with the aid of Figure \ref{fig:slopes},
\begin{center}
\begin{figure}[htb]
\includegraphics[scale=1]{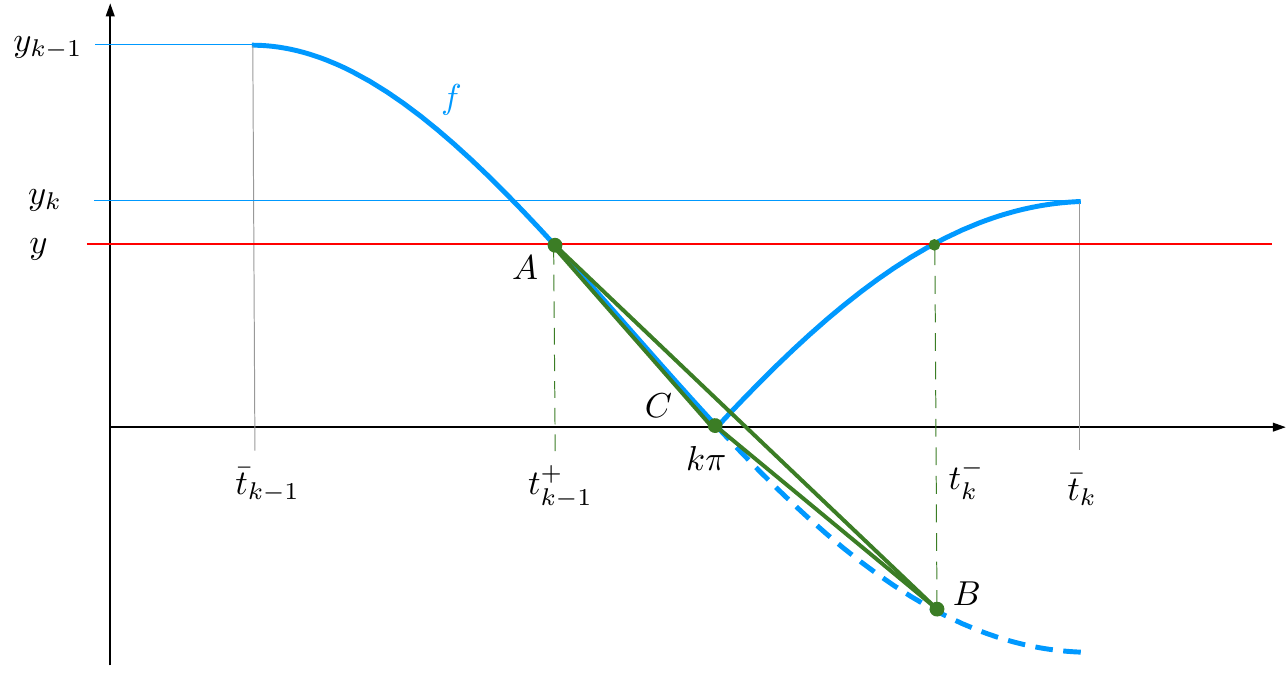}
\caption{The slope of the segment $AB$ is not smaller than the slope of either $AC$ or $BC$.}\label{fig:slopes}
\end{figure}
\end{center}
\begin{align*}
\frac{2y}{t_k^- - t_{k-1}^+} = \frac{\left|\frac{\sin t_k^-}{t_k^-} - \frac{\sin t_{k-1}^+}{t_{k-1}^+}\right|}{t_k^- - t_{k-1}^+} &= |\text{slope}(AB)| \\
&\geq \min\left\{|\text{slope}(AC)|, |\text{slope}(BC)|\right\}.
\end{align*}
{\red Let $\bar t_k \in (k\pi, (k+1)\pi)$ denote the point where $f(t)$ attains its local maximum $y_k$ on $(k\pi, (k+1)\pi)$.} Observe that
\[
|\text{slope}(BC)| = \frac{|\sin t_k^-|}{t_k^-(t_k^--k\pi)} \geq \frac{y_k}{\bar t_k - k\pi} \geq \frac{y_k}{\pi},
\]
where the first inequality follows from Lemma \ref{lm:slopes} (i) applied to $t_k^- < \bar t_k$. Similarly,
\[
|\text{slope}(AC)| = \frac{|\sin t_{k-1}^+|}{t_{k-1}^+(k\pi - t_{k-1}^+)} \geq \min\left\{\frac{y_{k-1}}{k\pi - \bar t_{k-1}}, \frac{1}{k\pi} \right\} \geq \min\left\{\frac{y_{k-1}}{\pi}, \frac{1}{k\pi} \right\},
\]
where in the first inequality we use Lemma \ref{lm:slopes} (ii) to lower bound the  function in question by the minimum of its values at the end-points $t = \bar t_{k-1}$ and $t = k \pi$. 
Finally, putting these two estimates together and using $y_k > \frac{1}{\pi (k+\frac{1}{2})}$, we obtain
\[
|\text{slope}(AB)| \geq \frac{1}{\pi^2(k+\frac{1}{2})}
\]
and, {\red consequently,}
\[
t_k^- - t_{k-1}^+ {\red = \frac{2y}{|\text{slope}(AB)|} } \leq 2\pi^2y\left(k + \frac{1}{2}\right),
\]
which results in 
\[
F(y) < 2\pi^{-p+1}y \left(\frac{3}{2}0.98^{-p-1}+\frac{5}{2}1.97^{-p-1}+\sum_{k=3}^m \left(k+\frac{1}{2}\right)(k-1)^{-p-1}\right) + \frac{1}{p}(m\pi)^{-p}.
\]
Since $y > y_{m+1} > \frac{1}{(m+\frac{3}{2})\pi}$, and $m \geq 30$, we have
\[
\frac{1}{p}(m\pi)^{-p} < \frac{1}{p}\left(\frac{m+3/2}{m}\right)^py^p \leq \frac{1}{p}1.05^py^p.
\]
Moreover, since $y < y_{m} < \frac{1}{m\pi}$, we have (crudely), $m-1 < \frac{1}{\pi y}$ and bounding the sum using the integral, we obtain
\begin{align*}
\sum_{k=3}^m \left(k+\frac{1}{2}\right)(k-1)^{-p-1} &= \sum_{k=2}^{m-1} \frac{k+\frac{3}{2}}{k^{p+1}}  \\
&< \int_1^{m-1} \left(x^{-p}+\frac{3}{2}x^{-p-1}\right) \dd x \\
&< \frac{(\pi y)^{p-1} - 1}{1-p} + \frac{3(1-(\pi y)^p)}{2p}.
\end{align*}
Therefore, in order to have $F(y) < G(y)$, it suffices to guarantee that
\begin{align*}
2\pi^{-p+1}y \Bigg(\frac{3}{2}0.98^{-p-1}+\frac{5}{2}1.97^{-p-1} &+ \frac{(\pi y)^{p-1} - 1}{1-p} + \frac{3(1-(\pi y)^p)}{2p}\Bigg)  \\
&\qquad+ \frac{1}{p}1.05^py^p < \frac{1}{p}\left(-6\log y\right)^{-p/2}
\end{align*}
holds for every $0 < p < 1$ and $0 < y < \frac{1}{30\pi}$. Since $-y\log y$ is increasing for $y < \frac{1}{e}$, we have $-\log y < \frac{\log(30\pi)}{30\pi}\frac{1}{y}$ for $0 < y < \frac{1}{30\pi}$. By monotonicity, for $0 < p < 1$, we have $\frac{3}{2}0.98^{-p-1}+\frac{5}{2}1.97^{-p-1} < \frac{3}{2}0.98^{-1-1}+\frac{5}{2}1.97^{-1} < 3$. It remains to use Lemma \ref{lm:sum-p}. This shows that Condition (ii) holds and the proof of the lemma is complete. It remains to show Claims A and B.
\end{proof}

\begin{proof}[Proof of Claim A]
By the integral representation for {\red the} $p$-norm from Lemma \ref{lm:fourier},
\[
\kappa_pH(p,2) = \E|U_1+U_2|^p-\E\left|\sqrt{\frac{2}{3}}Z\right|^p = \frac{2^{p+1}}{(p+1)(p+2)} - \frac{1}{\sqrt{\pi}}\left(\frac{4}{3}\right)^{p/2}\Gamma\left(\frac{1+p}{2}\right).
\]
By Lemma \ref{lm:gamma-all-p}, it suffices to prove that $2^{p+1} > (p+2)\left(1+\frac{p(p+1)}{6}\right)$ for all $0 < p < 1$. The 3rd derivative of the difference changes sign once on $(0,1)$ from $-$ to $+$. The 2nd derivative is negative at the end-points $p=0$ and $p=1$, so it is negative on $(0,1)$ and hence the difference is concave. It vanishes at the end-points $p=0$ and $p=1$, which finishes the argument.
\end{proof}

\begin{proof}[Proof of Claim B]
Our argument is split into two steps: first we show that $H(p,1)$ increases with $p$ and then we estimate $H(0.6,1)$. For somewhat similar computations, but related to random signs, see Section 5 in \cite{Mo}. In Step 1, to numerically evaluate the integrals in question, we will frequently use that given $0 < a < b$ and an integer $m$, integrals of the form $\int_a^b (\sin t)t^{-m} \dd t$  can be efficiently estimated to an arbitrary precision by expressing them in terms of the trigonometric integral functions $\text{Si}, \text{Ci}$. The same applies to the integrals of the form $\int_a^b e^{-t^2}t^{q}\dd t$ with $0 < a < b \leq \infty$ and real $q$, thanks to reductions to the incomplete gamma function $\Gamma$ and the exponential integral $\text{Ei}$. {\red We recall that for $x > 0$, $s \neq 0, -1, -2, \dots$,
\begin{align*}
\text{Si}(x) &= -\int_x^\infty \frac{\sin t}{t} \dd t = -\frac{\pi}{2}-\sum_{k=1}^\infty \frac{(-1)^kx^{2k-1}}{(2k-1)(2k-1)!}, \\
\text{Ci}(x) &= -\int_x^\infty \frac{\cos t}{t}\dd t = \gamma + \log x + \sum_{k=1}^\infty (-1)^k\frac{x^{2k}}{2k(2k)!},\\
\text{Ei}(-x) &= -\int_x^\infty \frac{e^{-t}}{t} \dd t = \gamma + \log x + \sum_{k=1}^\infty \frac{(-x)^k}{k\cdot k!}, \\
\Gamma(s, x) &= \int_x^\infty t^{s-1}e^{-t} \dd t = \Gamma(s) - \sum_{k=0}^\infty \frac{(-1)^kx^{s+k}}{k!(s+k)}
\end{align*}
(here $\gamma = 0.57721..$ is the Euler-Mascheroni constant).
These series representations allow to obtain arbitrarily good numerical approximations to these integrals.}

In Step 2, all the numerical computations are reduced to integrals of the form $\int_a^b \frac{\dd t}{t^q}$ which are explicit.

\emph{Step 1: $\frac{\partial}{\partial p} H(p,1) > 0$, $0.6 < p < 1$.} We have,
\[
\frac{\partial}{\partial p} H(p,1) = \int_0^\infty (-\log t)\frac{g(t) - f(t)}{t^{p+1}} \dd t.
\]
We break the integral into several regions. Recall $g > f$ on $(0, \pi)$, by Lemma \ref{lm:f<g}. Thus, plainly,
\[
\int_0^1 (-\log t)\frac{g(t) - f(t)}{t^{p+1}} \dd t > 0.
\]
Moreover, $g-f$ changes sign from $+$ to $-$ exactly once on $(\pi, 4)$ at $t = 3.578..$. Let $t_0 = 3.57$. On $(1,t_0)$, using $t^{-p-1} = t^{1-p}t^{-2} \leq t_0^{1-p}t^{-2}$, we obtain
\[
\int_1^{t_0} (-\log t)\frac{g(t) - f(t)}{t^{p+1}} \dd t \geq t_0^{1-p}\int_1^{t_0} (-\log t)\frac{g(t) - f(t)}{t^{2}} \dd t > -0.0297\cdot t_0^{1-p},
\]
where in the last inequality we use $\log t \leq \log\frac{5}{2} + \frac{2}{5}(t-\frac{5}{2})$ (by concavity) and then estimate the resulting integrals. Now, 
\[
\int_{t_0}^\infty (-\log t)\frac{g(t) - f(t)}{t^{p+1}} \dd t = \int_{t_0}^\infty (\log t)\frac{f(t)}{t^{p+1}} \dd t - \int_{t_0}^\infty (\log t)\frac{g(t)}{t^{p+1}} \dd t.
\]
For $t > t_0$, $t^{-p-1} = t^{1-p}t^{-2} > t_0^{1-p}t^{-2}$ and for $k \geq 1$, $\log t \geq \ell_k(t)$ on $(k\pi, (k+1)\pi)$ with
\[
\ell_k(t) = \frac{(k+1)\pi-t}{\pi}\log(k\pi) + \frac{t-k\pi}{\pi}\log((k+1)\pi),
\]
thus
\begin{align*}
\int_{t_0}^\infty (\log t)\frac{f(t)}{t^{p+1}} \dd t   &\geq t_0^{1-p}\left(\int_{t_0}^{2\pi} \ell_1(t)\frac{-\sin t}{t^3} \dd t + \sum_{k=2}^n \int_{k\pi}^{(k+1)\pi} \ell_k(t)\frac{(-1)^k\sin t}{t^3} \dd t\right).
\end{align*}
For $n = 5$, this gives
\[
\int_{t_0}^\infty (\log t)\frac{f(t)}{t^{p+1}} \dd t  > 0.0437\cdot t_0^{1-p}.
\]
Finally, since $\log u \leq \frac{u}{e}$, $u > 0$, we have $\frac{\log t}{t^p} < \frac{1}{ep} < \frac{1}{0.6e} < 0.6132 < 0.6132\cdot t_0^{1-p}$, thus
\[
\int_{t_0}^\infty (\log t)\frac{g(t)}{t^{p+1}} \dd t \leq 0.6132\cdot t_0^{1-p}\int_{t_0}^\infty \frac{e^{-t^2/6}}{t} \dd t < 0.0127\cdot t_0^{1-p}.
\]
Putting these together yields
\[
\frac{\partial}{\partial p} H(p,1) > (0.0437 - 0.0297 - 0.0127)t_0^{1-p} = 0.0013\cdot t_0^{1-p} > 0.
\]

\emph{Step 2: $H(0.6, 1) > 0$.} We have,
\[
H(0.6, 1) = \int_0^\pi \frac{e^{-t^2/6}-\frac{\sin t}{t}}{t^{8/5}} \dd t + \int_{\pi}^\infty \frac{e^{-t^2/6}}{t^{8/5}} \dd t - \int_{\pi}^\infty \frac{|\sin t|}{t^{13/5}}\dd t.
\]
On $(0,\pi)$, we use Taylor's polynomials to bound the integrand,
\[
e^{-t^2/6}-\frac{\sin t}{t} > \sum_{k=0}^7 \frac{(-t^2/6)^k}{k!} - \sum_{k=0}^6 \frac{(-1)^kt^{2k}}{(2k+1)!}.
\]
Plugging this into the integral results in
\[
\int_0^\pi \frac{e^{-t^2/6}-\frac{\sin t}{t}}{t^{8/5}} \dd t > 0.0434.
\]
Using the incomplete Gamma function, 
\[
\int_{\pi}^\infty \frac{e^{-t^2/6}}{t^{8/5}} \dd t > 0.0184.
\]
Finally, 
\begin{align*}
\int_{\pi}^\infty \frac{|\sin t|}{t^{13/5}} \dd t &= \int_0^\pi (\sin t)\left(\sum_{k=1}^\infty \frac{1}{(t+k\pi)^{13/5}} \right)\dd t \\
&\leq \int_0^\pi (\sin t)\left(\sum_{k=1}^n \frac{1}{(t+k\pi)^{13/5}} \right)\dd t  +  \int_{(n+1)\pi}^\infty \frac{\dd t}{t^{13/5}}.
\end{align*}
We use Taylor's polynomial again, $\sin t \leq 1-\frac{1}{2}(t-\pi/2)^2 + \frac{1}{24}(t-\pi/2)^4$. Choosing $n=8$ gives
\[
\int_{\pi}^\infty \frac{|\sin t|}{t^{13/5}} \dd t < 0.0615.
\]
Adding up these estimates yields $H(0.6, 1) > 0.0434 + 0.0184 - 0.0615 = 0.0003$. 
\end{proof}

\section{Inductive argument}\label{sec:ind}

As explained in Section \ref{sec:proof}, Theorem \ref{thm:s>2} gives the following corollary (we use homogeneity to rewrite \eqref{eq:cor} in an equivalent form, better suited for the ensuing arguments). Recall $c_p = \|Z\|_p/\sqrt{3}$ and define
\[
\varphi_p(x) = (1+x)^{p/2}, \qquad x \geq 0.
\]
\begin{corollary}\label{cor}
Let $0 < p < 1$. For every $n \geq 2$ and real numbers $a_2, \dots, a_n$ with {\red $\sum_{j=2}^n a_j^2 \gr 1$ and $a_j^2\ls 1$ for every $j=2,\ldots,n$}, we have
\[
\E\left|U_1 + \sum_{j=2}^n a_jU_j\right|^p \geq c_p^p\cdot \varphi_p\left(\sum_{j=2}^n a_j^2\right).
\]
\end{corollary}
The goal here is to remove the restriction on the $a_j$'s. The key idea from \cite{NP} is to replace $\varphi_p$ with a pointwise \emph{larger} function, thereby strengthening the inequality and to proceed by induction on $n$. We use the function from \cite{NP}, 
\[
\Phi_p(x) = \begin{cases} \varphi_p(x), & x \geq 1, \\ 2\varphi_p(1) - \varphi_p(2-x), & 0 \leq x \leq 1. \end{cases}
\]
Even though this function changes from being convex to concave at $x=1$, it is designed to satisfy the following \emph{extended convexity} property on $[0,2]$, crucial for the proof.
\begin{lemma}[Nazarov-Podkorytov, \cite{NP}]\label{lem:NP-Phi}
For every $0 < p < 2$ and $a, b \in [0,2]$ with $a+b \leq 2$, we have
\[
\frac{\Phi_p(a)+\Phi_p(b)}{2} \geq \Phi_p\left(\frac{a+b}{2}\right).
\]
\end{lemma}
As in \cite{CKT},  in order to have certain algebraic identities, we run the argument for $\xi_1, \xi_2, \dots$, independent random vectors in $\R^3$ uniformly distributed on the centred unit Euclidean sphere $S^2$. Here $\scal{\cdot}{\cdot}$ and $\|\cdot\|$ is the standard inner product and the resulting Euclidean norm in $\R^3$, respectively. 

\begin{theorem}\label{thm:ind}
Let $0 < p < 0.69$. For every $n \geq 2$ and vectors $v_2, \dots, v_n$ in $\R^3$, we have
\begin{equation}\label{eq:ind}
\E\left|\scal{e_1}{\xi_1} + \sum_{j=2}^n \scal{v_j}{\xi_j}\right|^p \geq c_p^p\cdot \Phi_p\left(\sum_{j=2}^n \|v_j\|^2\right).
\end{equation}
Here $e_1 = (1,0,0)$, the unit vector of the standard basis.
\end{theorem}

Note $\scal{v_j}{\xi_j}$ has the same distribution as $\|v_j\|U_j$ {(\red by rotational invariance, $\scal{v_j}{\xi_j}$ has the same distribution as $\|v_j\|\scal{e_1}{\xi_j}$ and by the Archimedes' hat-box theorem, the projection $\scal{e_1}{\xi_j}$ is a uniform random variable on $[-1,1]$)}. Since $\Phi_p \geq \varphi_p$, this gives Theorem \ref{thm:cp} for $0 < p < 0.69$, thereby completing its proof. It remains to show Theorem \ref{thm:ind}, which is done by repeating almost verbatim the proof of Theorem 18 from \cite{CKT}. We repeat the argument for the convenience of the reader. To adjust the proof of the base case we will need the following lemma.

\begin{lemma}\label{lm:base}
For every $0 < x  <1$ and $0 < p < 0.69$, we have
\begin{align*}
\frac{(1+x)^{2+p}-(1-x)^{2+p}}{2(2+p)x} &> \frac{1+p}{\sqrt{\pi}}\Gamma\left(\frac{1+p}{2}\right)\left(\frac{2}{3}\right)^{p/2}\left(2^{1+p/2}-(3-x^2)^{p/2}\right) \\
&{\red = (\|Z\|_p/\sqrt{3})^p(1+p)\Phi_p(x^2).}
\end{align*}
\end{lemma}
\begin{proof}
We first observe that keeping only the first two terms in the binomial series expansion, we obtain
\[
\frac{(1+x)^{2+p}-(1-x)^{2+p}}{2(2+p)x} = \sum_{k=0}^\infty \frac{1}{p+2}\binom{p+2}{2k+1}x^{2k} > 1 + \frac{p(p+1)}{6}x^2,
\]
because all the terms are positive. It thus suffices to show that for every $0 < x < 1$ and $0 < p < 0.69$,
\[
1 + \frac{p(p+1)}{6}x + \frac{1+p}{\sqrt{\pi}}\Gamma\left(\frac{1+p}{2}\right)\left(\frac{2}{3}\right)^{p/2}\left((3-x)^{p/2}-2^{1+p/2}\right) > 0
\]
(we have replaced $x^2$ by $x$). By the evident concavity in $x$, it suffices to check that the inequality holds at the end-points $x=0$ and $x=1$ which follows from Lemmas \ref{lm:gamma-all-p} and \ref{lm:gamma-small-p}, respectively.
\end{proof}

{\red
\begin{proof}[Proof of Theorem \ref{thm:ind}]
For the case $n=2$, we need to show that for every $v\in\mathbb{R}^3$
\begin{equation}\label{eq:base}
\E|\langle e_1,\xi_1\rangle+\langle v,\xi_2\rangle|^p\gr c_p^p\Phi_p(\|v\|^2).
\end{equation}
We first reduce this claim to the case $\|v\|\ls 1$: If $\|v\|>1$ then due to rotational invariance
\begin{align*}
\E|\langle e_1,\xi_1\rangle+\langle v,\xi_2\rangle|^p &= \|v\|^p\E\left|\langle\frac{e_1}{\|v\|},\xi_1\rangle+\langle\frac{v}{\|v\|},\xi_2\rangle\right|^p\\
                          &= \|v\|^p\E|\langle v',\xi_1\rangle+\langle e_1,\xi_2 \rangle|^p,
\end{align*}
where $v'\in\mathbb{R}^3$ is such that $\|v'\|=\frac{1}{\|v\|}<1$. On the other hand, due to homogeneity,
\[
\Phi_p(\|v\|^2)=\phi_p(\|v\|^2)=\|v\|^p\phi_p(\|v'\|^2),
\]
so \eqref{eq:base} is equivalent to
\[
\E|\langle v',\xi_1\rangle+\langle e_1,\xi_2 \rangle|^p \gr c_p^p\phi_p(\|v'\|^2), \qquad \|v'\|\ls 1
\]
and since $\Phi_p(x)\gr \phi_p(x)$ for $x\in[0,1]$ it is indeed sufficient to restrict to the case $\|v\|\ls 1$.

In this case, we set $x:=\|v\|\ls 1$ and compute explicitly the left and right hand side of \eqref{eq:base} to deduce that
\begin{align*}
\E|\langle e_1,\xi_1\rangle+\langle v,\xi_2\rangle|^p &= \E|U_1+xU_2|^p 
= \frac{(1+x)^{2+p}-(1-x)^{2+p}}{2(1+p)(2+p)x}
\gr c_p^p\Phi_p(x^2)
\end{align*}
with the aid of Lemma \ref{lm:base}.

For the inductive step, let $n\in\mathbb{N}$ and assume that \eqref{eq:ind} holds for every $ v_2,\ldots, v_{n-1}\in\mathbb{R}^3$. We let $v_2,\ldots,v_n\in\mathbb{R}^3$, $x:=\sum_{k=2}^n\|v\|^2$ and distinguish between the following mutually exclusive cases.

Case (i): $\|v_k\|>1$ for some $2\ls k\ls n$. Then $x>1$ and the wanted inequality is
\[
\E\left|\sum_{k=1}^n \langle v_k,\xi_k \rangle\right|^p \gr c_p^p \left(\sum_{k=1}^n \|v_k\|^2\right)^{p/2}
\]
with $v_1=e_1$. For $k=1,\ldots,n$ we let $v_k'=\frac{v_k^\ast}{\|v_1^\ast\|}$, where $v_1^*,\ldots,v_n^*$ is any rearrangement of $v_1,\ldots,v_n$ with $\|v_k^\ast\|\gr\|v_{k+1}^\ast\|$ for every $k=1,\ldots,n-1$. Then $\|v'_1\|=1$ and $\|v_k'\|\ls 1$ for $k=2,\ldots,n$. Due to homogeneity and the fact that $\langle v_1',\xi_1\rangle$ has the same distribution as $\langle e_1,\xi_1\rangle$ it is enough to prove
\[
\E\left|\langle e_1,\xi_1\rangle + \sum_{k=2}^n \langle v_k',\xi_k\rangle \right|^p \gr c_p^p \Phi_p\left(\sum_{k=2}^n\|v_k'\|^2\right).
\]
This is done on the next cases.

Case (ii): $\|v_k\|\ls 1$ for every $2\ls k\ls n$ and $x\gr 1$. We then again have that $\Phi_p(x)=\phi_p(x)$, and the desired inequality \eqref{eq:ind} coincides with
\[
\E\left|\sum_{k=1}^n \langle v_k,\xi_k \rangle\right|^p \gr c_p^p\left(\sum_{k=1}^n \|v_k\|^2\right)^{p/2}.
\]
Note that here we have
\[
\max_{1\ls k\ls n}\|v_k\|=1\ls \frac{1}{2}\left(1+x \right)=\frac{1}{2}\sum_{k=1}^n \|v_k\|^2,
\]
and since the distribution of $\sum_{k=1}^n \langle v_k,\xi_k \rangle$ is identical to that of $\sum_{k=1}^n\|v_k\|U_k$ it is clear that this case is handled by Theorem \ref{thm:s>2}.

Case (iii): $\|v_k\|\ls 1$ for every $2\ls k\ls n$ and $x< 1$. We use the fact that $(\xi_{n-1},\xi_n)$ has the same distribution as $(\xi_{n-1},Q\xi_{n-1})$ where $Q$ is a random orthogonal matrix independent of all the $\xi_k$'s to write
\begin{align*}
&\E\left|\langle e_1,\xi_1\rangle +\sum_{k=2}^n\langle v_k,\xi_k \rangle\right|^p = \E|\langle e_1,\xi_1\rangle + \langle v_2,\xi_2\rangle+\ldots+\langle v_{n-1},\xi_{n-1}\rangle+\langle Q^\top v_n,\xi_{n-1}\rangle|^p\\
                                   &\hspace{55pt}= \E_Q\left[\E_{(\xi_k)_{k=2}^{n-1}}|\langle e_1,\xi_1\rangle + \langle v_2,\xi_2\rangle+\ldots+\langle v_{n-2},\xi_{n-2}\rangle+\langle Q^\top v_n,\xi_{n-1}\rangle|^p\right].
\end{align*}
By the inductive hypothesis applied to $(v_2,\ldots,v_{n-2},v_{n-1}+Q^\top v_n)$ (conditioned on the value of $Q$) we get
\[
\E\left|\langle e_1,\xi_1\rangle +\sum_{k=2}^n\langle v_k,\xi_k \rangle\right|^p \gr c_p^p\E_Q\Phi_p(\|v_2\|^2+\ldots+\|v_{n-2}\|^2+\|v_{n-1}+Q^\top v_n\|^2).
\]
Finally note that
\begin{align*}
&\E_Q\Phi_p\left(\sum_{k=2}^{n-2}\|v_k\|^2+\|v_{n-1}+Q^\top v_n\|^2\right)=\\
&\hspace{55pt}= \E_Q\frac{\Phi_p(x+2\langle v_{n-1}+Q^\top v_n\rangle)+\Phi_p(x-2\langle v_{n-1}+Q^\top v_n\rangle)}{2} \gr \Phi_p(x)
\end{align*}
by the symmetry of $\Phi_p$ and Lemma \ref{lem:NP-Phi} (applied for $a=x+2\langle v_{n-1}+Q^\top v_n\rangle$ and $b=x-2\langle v_{n-1}+Q^\top v_n\rangle$ which satisfy $a+b=2x\ls 2$). This concludes the proof of the inductive step.
\end{proof}
}

\section{R\'enyi entropy: Proof of Theorem \ref{thm:ent}}\label{sec:ent}

For the lower bound,
\[
h_p\left(\sum_j a_jU_j\right) \geq h_1\left(\sum_j a_jU_j\right) \geq h_1(U_1),
\]
where the first inequality follows from the fact that $p \mapsto h_p(\cdot)$ is nonincreasing and the second one is justified by the entropy power inequality (see, e.g. Theorem 4 in \cite{DCT}). It remains to notice that $h_p(U_1) = \log 2$ for every $p$.

Towards the upper bound, we first note that for nonnegative functions $f$ and $g$, $0 <p < 1$, we have
\[
\left(\int f^p\right)^{\frac{1}{p}}\left(\int g^p\right)^{\frac{p-1}{p}} \leq \int fg^{p-1}.
\]
This follows directly from H\"older's inequality. Now, fix a unit vector $a$ in $\R^n$, let $f$ be the density of $\sum_j a_jU_j$ and $g(x) = (2\pi/3)^{-1/2}e^{-x^2/6}$, the density of $Z/\sqrt{3}$. In view of the above inequality, it suffices to show that
\[
\int fg^{p - 1} \leq \int gg^{p-1}.
\]
Since 
\[
g(x)^{p-1} = (2\pi/3)^{\frac{1-p}{2}}\sum_{k=0}^\infty \frac{1}{k!} \left(\frac{1-p}{6}\right)^kx^{2k},
\]
it suffices to show that for each positive integer $k$,
\[
\E\left(\sum a_jU_j\right)^{2k} = \int x^{2k}f(x) \dd x \leq \int x^{2k} g(x) \dd x = \E\left(\frac{Z}{\sqrt{3}}\right)^{2k}.
\]
This follows from the main result of \cite{LO}, that $C_p = \|Z\|_p/\sqrt{3}$, $p > 1$, see \eqref{eq:LO}.\hfill$\square$

We finish by remarking that the problem of maximising $h_p(\sum a_jU_j)$ under a variance constraint for a \emph{fixed} number of summands to the best of our knowledge remains wide open for $p \in (0,\infty)$. The case of Shannon entropy, $p=1$, seems to be the most important and interesting, see Question 9 in \cite{ENT0}, or Question 3 in \cite{BNZ}, also comprehensively presenting many other related and tangential problems. The natural conjecture is that: $h_1(\sum_{j=1}^n a_jU_j) \leq h_1(\sum_{j=1}^n \frac{1}{\sqrt{n}}U_j)$, for every unit vector $a$ in $\R^n$ (see 8.3.1 in \cite{BNZ} for a conceivable approach). The case $p=0$ is of course trivial, whereas the case $p=\infty$ amounts to the cube-slicing inequalities: $h_\infty(\sum_{j=1}^n a_jU_j) \leq h_\infty(U_1)$ is due to Hadwiger and, independently, Hensley (see \cite{Ha, He}), $h_\infty(\sum_{j=1}^n a_jU_j) \geq h_\infty((U_1+U_2)/\sqrt{2})$ is due to Ball (see \cite{B}).

\end{document}